\documentclass[a4paper, 10pt]{amsart}

\usepackage{amsmath,amsthm,amssymb}
\usepackage{mathtools}
\usepackage{enumerate}
\usepackage{ulem}
\usepackage[driverfallback=dvipdfm]{hyperref}
\usepackage[dvipdfmx]{graphicx}
\usepackage{bbold}
\usepackage{xcolor}

\usepackage{dsfont}
\usepackage{comment}
\usepackage{tikz-cd}

\theoremstyle{plain}
	\newtheorem{thm}{Theorem}[section]
	\newtheorem{cor}{Corollary}[section]
	\newtheorem{lem}{Lemma}[section]
	\newtheorem{prop}{Proposition}[section]
	\newtheorem{rmk}{Remark}
    \newtheorem{prob}{Problem}

\theoremstyle{definition}

\makeatletter
	
	\@addtoreset{equation}{section}
\makeatother

\def\A{\text{AC}}

\title[Isometries on ordered Banach spaces of absolutely continuous functions]{A variant of Tingley's problem on ordered Banach spaces of absolutely continuous functions}

\author[M.--R. Lin]{Min--Ruei Lin}
\address[M.--R. Lin]
{Department of Applied Mathematics,
National Sun Yat--sen University,
Kaohsiung, 80424, Taiwan;
And
Graduate School of Science and Technology, Niigata University, Niigata 950--2181, Japan}
\email{m082030021@student.nsysu.edu.tw}



\begin{document}

\begin{abstract}
For each $1\le p\le\infty$ and $j=1,2$, let $AC^p(\Omega_j)$ denote the Banach space of complex-valued absolutely continuous functions on a closed unit interval $\Omega_j=[x_j,x_j+1]$.
We equip $AC^p(\Omega_j)$ with the $p$--norm $\|f\|_{\A,p}$, and the order $\ge_{\A}$ defined by $f(x_j)\ge0$ and $f'\ge 0$ a.e.
Set
$$S(AC^p(\Omega_j))^+=\{f\in AC^p(\Omega_j):\|f\|_{\A,p}=1,\  f\ge_{\A}0\}.$$
We prove that, for each $1\le p\le\infty$, every surjective isometry $S(AC^p(\Omega_1))^+\to S(AC^p(\Omega_2))^+$ extends uniquely to a complex--linear isometric order isomorphism from $AC^p(\Omega_1)$ onto $AC^p(\Omega_2)$.
As an application, we obtain a corresponding extension theorem for surjective phase--isometries.
\end{abstract}

\maketitle

\section{Introduction}
Tingley's problem~\cite{T87} asks whether every surjective isometry $T:S(E)\to S(F)$ between the unit sphere of two Banach spaces extends uniquely to a surjective linear isometry between the Banach spaces $E$ and $F$.
The problem has attracted considerable attention for function spaces~\cite{cue1, cue2, hat1, hat2, hat3, hiro, wang1, wang2, wang3, tan1, tan2}.
To date, no counterexample is known.
In particular, Tingley's problem has a positive answer for $L^p(\mu)$--spaces, when $\mu$ is $\sigma$--finite~\cite{tan1,tan2}.
However, results on absolutely continuous function spaces remain limited~\cite{Hos18, Pat82}.

A related direction is to replace the whole unit sphere by its positive part.
This is natural for ordered Banach spaces whose positive cones generate the underlying real ordered structure; in the complex case, the whole space is then obtained by complexification.
Peralta~\cite{PAM} proposed a variant of Tingley's problem.
An extended form of Peralta's problem states as follows:
\begin{prob}\label{Tingley}
    Let $T:S(E)^+\to S(F)^+$ be a surjective isometry between the positive unit spheres of two ordered Banach spaces $E$ and $F$ with generating cones; namely
    $$\|T(x)-T(y)\| = \|x-y\|\qquad (x,y\in S(E)^+).$$
    Is there a surjective linear isometry $\Phi: E\to F$ extending $T$?
\end{prob}

Later, Leung, Ng and Wong~\cite{LNW21} solved this problem affirmatively for $L^p$--spaces with $1\le p\le\infty$.
In particular, this work is being applied directly to complex $L^p$--spaces and its extension is complex-linear.
In 2026, Enami \textit{et al.}~\cite{C1LIP} demonstrated a representation method for the cases of $C^1[0,1]$ and $Lip[0,1]$.
More precisely, they consider isometric isomorphisms from $C^1[0,1]$ onto $\mathbb{C}\oplus C([0,1])$ and from $Lip[0,1]$ onto $\mathbb{C}\oplus L^\infty([0,1])$, respectively, and characterize geometric properties of the positive unit ball in terms of their diameters.
They apply Mankiewicz’s theorem to surjective isometries between convex bodies $B(C[0,1])^+$ and $B(L^\infty[0,1])^+$, respectively, to obtain affirmative answers to Problem~\ref{Tingley}. 
However, $B(L^{1}([0,1]))^+$ has a different geometric structure.
The argument in~\cite[Theorem~1.1]{C1LIP} relies on a positive unit ball of diameter $1$ with nonempty interior.
In contrast, $\operatorname{diam} B(L^{1}([0,1]))^+=2$,
and $B(L^{1}([0,1]))^+$ has empty interior in $L^1([0,1])$.
Thus, even though $AC([0,1])$ admits the identification $AC([0,1])\simeq\mathbb{C}\oplus L^{1}([0,1])$, its associated $L^1$--geometry places it outside the scope of the convex-body approach developed in~\cite{C1LIP}.

In this paper, we prove that Problem~\ref{Tingley} has a positive answer for the spaces $AC^p(\Omega)$ of complex--valued absolutely continuous functions on a closed unit interval $\Omega$, equipped with the $p$--norm for every $1\leq p\leq \infty$.
Although we use the same representation idea as in~\cite{C1LIP} for cases $1<p\le\infty$, the subsequent argument is different.
On the other hand, for $p=1$, we represent $AC^1(\Omega)$ as an $L^1(X,\mu)$ over the disjoint union $X=\{\ast\}\sqcup\Omega$, equipped with a suitable $\sigma$-algebra and a
combined measure $\mu$.

The paper is organized as follows.
In Section 2, we introduce the notations and definitions used throughout the paper.
Several preliminary propositions and corollaries are established in Section 2.
In Section 3, we state the main theorem as well as necessary lemmas for characterizations of the corresponding geometry, and provide the proof of the main theorem at the end of the section.
In the last section, the main theorem is utilized for an application
of phase--isometries.
\section{Notations and preliminaries}\label{sec2}
We present some notations used in this paper.
Let $\Omega=[x_0,x_0+1]$ be equipped with the Lebesgue measure $m$, where $x_0\in\mathbb{R}$ is the left endpoint of the interval $\Omega$.
Let $AC_{\mathbb{C}}(\Omega)$ ($AC_{\mathbb{R}}(\Omega)$, resp.) be the vector space of complex--valued (real--valued, resp.) absolutely continuous functions on $\Omega$. 
We define the positive cone of $AC_{\mathbb{C}}(\Omega)$ by
$$AC_\mathbb{C}(\Omega)^+:=\{f\in AC_\mathbb{R}(\Omega): f(x_0)\ge0,\ f'\ge0\ \text{a.e.}\} = AC_\mathbb{R}(\Omega)^+,$$
and its induced order $\le_{\A}$.
We consider the following $p$--norm
\begin{equation}\label{p_norm}
    \|f\|_{\A,p}:=
    \begin{cases}
        (|f(x_0)|^p + \|f'\|^p_1)^{1/p}\quad (1\le p< \infty) \\
        \max\{|f(x_0)|,\ \|f'\|_1\}\quad (p=\infty)
    \end{cases}
    \qquad (f\in AC_\mathbb{C}(\Omega)).
\end{equation}
Then $(AC_\mathbb{C}(\Omega),\ \|\cdot\|_{\A,p},\ \le_{\A})$ is an ordered Banach space.
We simply denote $(AC_\mathbb{C}(\Omega),\ \|\cdot\|_{\A,p},\ \le_{\A})$ by $AC^p(\Omega)$.
The positive part of the unit sphere of $AC^p(\Omega)$ is defined as 
$$S(AC^p(\Omega))^+:= S(AC^p(\Omega))\cap AC_\mathbb{C}(\Omega)^+=\{f\in AC_\mathbb{C}(\Omega)^+ : \|f\|_{\A,p} = 1\},$$
where $S(AC^p(\Omega))=\{f\in AC^p(\Omega): \|f\|_{\A,p} = 1\}$ is the unit sphere of $AC^p(\Omega)$.
On the other hand, we use an almost everywhere derivative operator $\mathfrak{D}:AC^p(\Omega)\to L^1(\Omega)$ defined by $\mathfrak{D}(f):=f'$ a.e..
If $\mathfrak{D}|_{S(AC^p(\Omega))^+}^{-1}$ exists, we denote $\mathfrak{I}:=\mathfrak{D}|_{S(AC^p(\Omega))^+}^{-1}$.
The positive unit ball of $L^1(\Omega)$ is written as 
$B(L^1(\Omega))^+=\{u\in L^1_\mathbb{C}(\Omega)^+:\|u\|_1\le 1\}$, where $L^1_\mathbb{C}(\Omega)^+ = L^1_\mathbb{R}(\Omega)^+$ is the positive cone of the vector space $L^1_\mathbb{R}(\Omega)$ with usual order.
For each $1\le p\le\infty$, a positive cone of $\mathbb{C}\oplus L_\mathbb{C}^1(\Omega)$ is defined as 
$\mathbb{R}^+\oplus L_\mathbb{R}^1(\Omega)^+:=\{(a,u)\in \mathbb{R}\oplus L_\mathbb{R}^1(\Omega) :a\ge0, u\ge 0\ \text{a.e.}\}$, and its induced order $\le_{\oplus}$.
Define 
\begin{equation}
    \|(a,u)\|_{\oplus,p}:=
    \begin{cases}
        (|a|^p + \|u\|^p_1)^{1/p}\quad (1\le p< \infty) \\
        \max\{|a|,\ \|u\|_1\}\quad (p=\infty)
    \end{cases}\qquad((a,u)\in\mathbb{C}\oplus L^1_\mathbb{C}(\Omega)).
\end{equation}
We write $\mathbb{C}\oplus_p L^1(\Omega)$ for the ordered Banach space $(\mathbb{C}\oplus L_\mathbb{C}^1(\Omega),\ \|\cdot\|_{\oplus, p},\ \le_{\oplus})$.
The positive unit sphere of $\mathbb{C}\oplus_p L^1(\Omega)$ is written as 
$\Sigma^+=\{(a,u)\in\mathbb{R}^+\oplus L_\mathbb{R}^1(\Omega)^+: \|(a,u)\|_{\oplus,p} =1\}$.

We now establish several preliminary results.
\begin{prop}\label{prop:2.1}
    Suppose $1\le p\le \infty$.
    Then $AC^p(\Omega)$ is complex--linearly isometrically order isomorphic to $\mathbb{C}\oplus_p L^1(\Omega)$.
\end{prop}
\begin{proof}
For $1\le p\le\infty$, define a map
$$J:AC^p(\Omega)\to \mathbb{C}\oplus_p L^1(\Omega),\qquad J(f) = (f(x_0),f').$$
Complex--linearity is immediate, and bijectivity follows from the fundamental theorem of calculus for absolutely continuous functions.
Since $J$ is complex--linear and bijective, isometry follows as norm-preserving follows.
That is, for each $f\in AC^p(\Omega)$, we have $\|J(f)\|_{\oplus,p}:=(|f(x_0)|^p+\|f'\|^p_1)^{1/p} = \|f\|_{\A,p}$ when $1\le p<\infty$, and $\|J(f)\|_{\oplus,\infty}=\max\{|f(x_0)|,\|f'\|_1\}=\|f\|_{\A,\infty}.
$ when $p=\infty$.
It remains to prove that $J$ preserves orders from both directions.
Because of the linearity and bijectivity, it is sufficient to show that $f\in AC_\mathbb{R}(\Omega)^+$ if and only if $J(f)\in \mathbb{R}^+\oplus L_\mathbb{R}^1(\Omega)^+$. 
Let $f\in AC_\mathbb{R}(\Omega)^+$, then
\begin{align*}
f\ge_{\A} 0\quad &\Longleftrightarrow\quad f(x_0)\ge0,\ f'\ge 0\text{ a.e.}\quad
\Longleftrightarrow\quad (f(x_0),f')\in \mathbb{R}^+\oplus L^1(\Omega)^+\quad
\Longleftrightarrow\quad J(f)\ge_{\oplus} 0
\end{align*}
That is, both $J$ and $J^{-1}$ preserve orders, and hence $J$ is order-preserving.
\end{proof}
For $1\le p<\infty$, the preceding identification allows us to describe the derivative map on
$S(AC^p(\Omega))_+$ in terms of the second--coordinate projection.

\begin{cor}\label{cor:2.1}
    Suppose $1\le p<\infty$.
    Then the restriction $\mathfrak{D}|_{S(AC^p(\Omega))^+}:S(AC^p(\Omega))^+\to B(L^1(\Omega))^+$ is bijective, and its inverse is given by
    $$\mathfrak{I}(u)(x)=(1-\|u\|_1^p)^{1/p} + \int_{x_0}^x u(t) dt,\quad x\in\Omega,\quad u\in B(L^1(\Omega))^+.$$
\end{cor}
\begin{proof}
By Proposition~\ref{prop:2.1}, we have a complex--linear isometric order isomorphism $J:AC^p(\Omega)\to \mathbb{C}\oplus_p L^1(\Omega)$ satisfying $J(S(AC^p(\Omega))^+)=\Sigma^+$.
Under this identification, the derivative map $\mathfrak{D}$ corresponds to the
second--coordinate projection
$$
\pi_2:\Sigma^+\to B(L^1(\Omega))^+,
\qquad
\pi_2(a,u)=u.
$$

We show that $\pi_2$ is bijective.
If $(a,u)\in\Sigma^+$, then $u\ge 0$ a.e. and $\|u\|_1^p\le a^p+\|u\|_1^p=1$.
Thus $u\in B(L^1(\Omega))^+$.
Conversely, if $u\in B(L^1(\Omega))^+$, then $a=\left(1-\|u\|_1^p\right)^{1/p}$ is the unique nonnegative scalar satisfying $a^p+\|u\|_1^p=1$.
Hence $(a,u)\in\Sigma^+$, and $\pi_2(a,u)=u$.
Therefore $\pi_2:\Sigma^+\to B(L^1(\Omega))^+$ is bijective, with inverse
$$u\mapsto\left(\left(1-\|u\|_1^p\right)^{1/p},u\right).$$

Since $\mathfrak{D}|_{S(AC^p(\Omega))^+}=\pi_2\circ J|_{S(AC^p(\Omega))^+}$, it follows that $\mathfrak{D}|_{S(AC^p(\Omega))^+}$ is bijective as $J$ and $\pi_2$ are.
Then $\mathfrak{D}|_{S(AC^p(\Omega))^+}^{-1} = J|_{S(AC^p(\Omega))^+}^{-1}\circ\pi_2^{-1}$.
Namely, $\mathfrak{I}(u)=J^{-1}\left(\left(1-\|u\|_1^p\right)^{1/p},u\right)$ for each $u\in B(L^1(\Omega))^+$.
By the definition of $J^{-1}$, this gives
$$
\mathfrak{I}(u)(x)=\left(1-\|u\|_1^p\right)^{1/p}+\int_{x_0}^x u(t)\,dt,
\qquad x\in\Omega.
$$
This completes the proof.
\end{proof}

Although Corollary~\ref{cor:2.1} applies when $p=1$, the later argument for
$1<p<\infty$ will not be used in this case.
Instead, for $p=1$, the direct sum $\mathbb C\oplus_1 L^1(\Omega)$ admits an especially simple realization as $L^1(X,\mathfrak{A},\mu)$.

\begin{prop}\label{prop:2.2}
    Let $\ast$ be an isolated point disjoint from $\Omega$.
    Let $X=\{\ast\}\sqcup\Omega$, a $\sigma$-algebra $\mathfrak{A}:=\{A\subseteq X: A\cap\Omega\in\mathcal{L}(\Omega)\}$, and a combined measure $\mu(E):=m(E\cap\Omega)+\delta_{\ast}(E)$ for each $E\in\mathfrak{A}$.
    Then $\mathbb{C}\oplus_1 L^1(\Omega)$ is complex--linearly isometrically order isomorphic to $L^1(X,\mathfrak{A},\mu)$.
\end{prop}
\begin{proof}
Recall that $\mu$ is a measure on $X$~\cite[Exercise 4, Chapter 17.1]{hlr4}, and hence we can define $L^1(X,\mathfrak{A},\mu)$.
Define
$$U:\mathbb{C}\oplus_1 L^1(\Omega)\to L^1(X,\mathfrak{A},\mu),\qquad U(a,u) = a\chi_{\{\ast\}}+u\chi_\Omega.$$
Suppose $[u]=[v]\in L^1(\Omega)$, then equivalently $m\{t\in\Omega:u(t) \neq v(t)\}=0$.
Since $\ast\not\in\{t\in\Omega:u(t)\neq v(t)\}$, we have
$a\chi_{\{\ast\}} + u\chi_\Omega=a\chi_{\{\ast\}}+ v\chi_\Omega$ $\mu$-a.e.,
and thus $U(a,[u]) = U(a, [v])$.
Hence $U$ is well-defined.
complex--linearity holds immediately.
Since $U$ is complex--linear, to show isometry, it is sufficient to show $U$ is norm-preserving.
Let $(a,u)\in\mathbb{C}\oplus_1 L^1(\Omega)$, then
\begin{equation*}
\|U(a,u)\|_{L^1(X)} = \int_{X}|a\chi_{\{\ast\}} + u\chi_{\Omega}| d\mu=\int_{\{\ast\}}|a|d\delta_{\ast}+\int_{\Omega}|u| dm=|a|+\|u\|_{1}=\|(a,u)\|_{\oplus,1},
\end{equation*}
thus $U$ is an isometry.
Now choose a representative $\widetilde{w}$ of $w\in L^1(X,\mathfrak{A},\mu)$, then $\widetilde{w}=a\chi_{\{\ast\}}+u\chi_\Omega$ for some $a\in\mathbb{C}$ and $u\in L^1(\Omega)$ because of $X=\{\ast\}\sqcup\Omega$.
Then $w= U(a,u)$, and hence $U$ is surjective.

Finally, we demonstrate $U$ preserves the positive cones in both directions.
Since $U$ is linear, it is sufficient to show that $(a,u)\ge_{\oplus}0$ if and only if $U(a,u)\ge 0$ $\mu$-a.e..
Let $(a,u)\in\mathbb{R}^+\oplus L_\mathbb{R}^1(\Omega)^+$, then
\begin{align*}
(a,u)\ge_{\oplus}0\quad&\Longleftrightarrow\quad a\ge0\quad\text{and}\quad u\ge 0\ m\text{--a.e.}\\
&\Longleftrightarrow\quad a\chi_{\{\ast\}}+u\chi_\Omega\ge 0\ \mu\text{--a.e.}\\
&\Longleftrightarrow\quad U(a,u)\ge 0\ \mu\text{--a.e.}
\end{align*}
That is, both $U$ and $U^{-1}$ preserve orders, and thus $U$ is order-preserving.
\end{proof}

\begin{cor}\label{cor:2.2}
    For $j=1,2.$
    Let $\Phi:L^1(X_1,\mathfrak{A}_1,\mu_1)\to L^1(X_2,\mathfrak{A}_2,\mu_2)$ be a complex--linear isometric order isomorphism, where the measure spaces $(X_j,\mathfrak{A}_j, \mu_j)$ are defined in Proposition~\ref{prop:2.2}.
    Then $\Phi(\chi_{\{\ast_1\}}) = \chi_{\{\ast_2\}}$.
\end{cor}
\begin{proof}
We simply write $L^1(X_j,\mathfrak{A}_j, \mu_j)$ as $L^1(\mu_j)$.
Consider a nonnegative representative $h$ of $\Phi(\chi_{\{\ast_1\}})\in L^1(\mu_2)$.
Put $S:=\{t\in X_2: h(t)>0\}$, the support of $h$.
Let $E\subseteq S$ be measurable, then we have $0\le h\chi_E\le h$ $\mu_2$--a.e., and thus
\begin{equation}\label{eq:2.1}
h\chi_E = \lambda h\quad \mu_2\text{--a.e.},
\end{equation}
for some $\lambda\in [0,1]$.
Since $\chi_E$ takes values either $1$ or $0$, it follows that $\lambda \in \{0,1\}.$
If $\lambda = 0$, then equality~(\ref{eq:2.1}) reveals $\chi_E = 0$ $\mu_2\text{--a.e.}$ on $S$.
Equivalently, $\mu_2(E) = 0$.
If $\lambda=1$, then equality~(\ref{eq:2.1}) becomes $h\chi_E = h$ $\mu_2\text{--a.e.}$
Then $\mu_2(E) = \mu_2(S)$, and thus $S$ is the atom in $(X_2,\mathfrak{A}_2, \mu_2)$.
the unique atom of the measure space $(X_2,\mathfrak{A}_2, \mu_2)$ is $\{\ast_2\}$, we have $\Phi(\chi_{\{\ast_1\}}) = c\chi_{\{\ast_2\}}$ for some $c>0$.
By the isometry of $\Phi$, we obtain $c=1$, and thus $\{\chi_{\{\ast_j\}}\}$ sends to each other by $\Phi$.
We conclude the required statement.
\end{proof}

\section{Main results}
Throughout this section, let $\Omega_1=[x_1,x_1+1]$ and $\Omega_2=[x_2,x_2+1]$ be closed intervals in $\mathbb{R}$, both equipped with Lebesgue measure $m$.
Let $\Delta:S(AC^p(\Omega_1))^+\to S(AC^p(\Omega_2))^+$ be a surjective isometry.
We state the main theorem
\begin{thm}\label{thm:main}
Suppose $1\le p\le \infty$.
Let $\Delta:S(AC^p(\Omega_1))^+\to S(AC^p(\Omega_2))^+$ be a surjective isometry.
Then there exists a unique complex--linear isometric order isomorphism $\widetilde{\Delta}:AC^p(\Omega_1)\to AC^p(\Omega_2)$ such that $\widetilde{\Delta}|_{S(AC^p(\Omega_1))^+} = \Delta$.
In particular, 
$$\widetilde{\Delta}(f)(x)=f(x_1)+\int_{x_2}^x \Lambda(f')(t) dt,\qquad f\in AC^p(\Omega_1),\quad x\in\Omega_2,$$
where $\Lambda:L^1(\Omega_1)\to L^1(\Omega_2)$ is the unique complex--linear isometric order isomorphism determined by $\Delta$.
\end{thm}
At the first stage, we collect several metric facts needed for the cases $1<p\leq\infty$.

\begin{lem}\label{lem:3.1}
Suppose $1<p\le\infty$.
Then $\Delta(\mathds{1}_1) = \mathds{1}_2$. 
\end{lem}
\begin{proof}
For each $f\in S(AC^p(\Omega))^+$, inspired by~\cite{LNW25}, we consider a radius function defined as 
$$\rho(f):=\sup_{g\in S(AC^p(\Omega))^+}\|f-g\|_{\A,p},$$
By Proposition~\ref{prop:2.1}, we write $f$ in the form $(a,u)\in \Sigma^+$.
Firstly, we consider $\rho(\mathds{1})$.
For $g=(b,v)\in \Sigma^+$,
$$
\begin{aligned}
\|\mathbb{1}-g\|_{\A,p}
&=
\begin{cases}
\bigl((1-b)^p+\|v\|_1^p\bigr)^{1/p}
=
\bigl((1-b)^p+1-b^p\bigr)^{1/p}
\le 2^{1/p}
& (1<p<\infty) \\[0.8ex]
\max\{1-b,\|v\|_1\}\le 1
& (p=\infty)
\end{cases}.
\end{aligned}
$$
Hence $\rho(\mathds{1})=2^{1/p}$ for $1<p<\infty$, and $\rho(\mathds{1}) = 1$ for $p=\infty$.
Now let $f=(a,u)\in \Sigma^+\setminus\{(1,0)\}$.
Since $\Omega$ is non-atomic, we may choose measurable sets $E_n\subsetneq\Omega$ such that
$0<m(E_n)<1/n$.
Define $v_{n}:=\chi_{E_n}/m(E_n)\ge 0\text{ a.e.}$.
Then $\|v_n\|_1 = 1$, and hence $v_n\in S(L^1(\Omega))^+$.
We have 
\begin{align*}
\|u-v_n\|_1 &= \|u\|_1+\|v_n\|_1-2\int_{\Omega} (u\wedge v_n) \\
&= \|u\|_1+1-2\int_{E_n} (u\wedge\frac{1}{m(E_n)}) \\
&\ge \|u\|_1+1-2\int_{E_n} u.
\end{align*}
Thus, $\|u-v_n\|_1\to \|u\|_1+1$ as $n\to\infty$.
Let $g_{n}=(0,v_{n})\in \Sigma^+$.
In the case $1<p<\infty$, we have $\rho(f)^p\ge\limsup_{n}\|f-g_n\|^p_{\A,p}=a^p+(\|u\|_1+1)^{p}=1-\|u\|_1^{p}+(\|u\|_1+1)^{p}$.
Since $p>1$ and $\|u\|_1>0$, we have $(\|u\|_1+1)^{p} >1+\|u\|_1^{p}$, and thus $\rho(f)^{p}>2$.
Therefore $\rho(f)>2^{1/p}=\rho(\mathds{1})$.
In the case $p=\infty$, $\rho(f)\ge\limsup_{n}\|f-g_n\|_{\A,\infty} =\limsup_{n}\|u-v_n\|= \|u\|_1+1$.
Since $f\neq (1,0)$, and hence $\|u\|_1> 0$, we conclude $\rho(f)  >1=\rho(\mathds{1})$.

Consequently, $\mathds{1}$ is the unique minimizer of the radius function $\rho$ in $S(AC^p(\Omega))^+$ for each $1<p\le\infty$. 
Since $\Delta$ is a surjective isometry, it preserves
$\rho$.
Hence $\rho(\Delta(\mathds{1}_1))=\rho(\mathds{1}_{2})$.
By uniqueness, $\Delta(\mathds{1}_1)=\mathds{1}_2$.
\end{proof}

\begin{rmk}\label{rmk:rho}
    For $p=\infty$, we have $\rho(a,u) = \|u\|_1 +1$ for each $(a,u)\in \Sigma^+$.
\end{rmk}

Having identified the constant function $\mathds{1}$ as a distinguished point of $S(AC^p(\Omega))^+$ for $1<p\leq \infty$, we next recover the $L^1$-component from metric structure.
The following lemma records separation properties of $B(L^1(\Omega))^+$ and $S(L^1(\Omega))^+$.

\begin{lem}\label{lem:3.2}
    (a) Let $u,v\in B(L^1(\Omega))^+$ with $\|u\|_1 = \|v\|_1$.
    If $\|u-z\|_1 = \|v-z\|_1$ for all $z\in S(L^1(\Omega))^+$, then $u=v$.\newline
    (b) Let $u,v\in S(L^1(\Omega))^+$ and $c\in[0,2)$.
    If $\max\{c,\|u-z\|_1\} = \max\{c,\|v-z\|_1\}$ for all $z\in S(L^1(\Omega))^+$, then $u=v$.
\end{lem}
\begin{proof}
We prove both assertions by contraposition.
(a) Suppose $u\ne v$.
Put $f=u-v$.
Then $f\ne 0$ on a set of positive measure.
Define
$$\Omega_k=\{t\in\Omega:u(t)\le k,\ v(t)\le k\}.$$
Then $\Omega_k\uparrow\Omega$ up to a null set.
Choose $k\in\mathbb N$ such that $m\{t\in\Omega_k:f(t)>0\}>0$, or else such that $m\{t\in\Omega_k:f(t)<0\}>0$.

Assume that $m\{t\in\Omega_k:f(t)>0\}>0$; the other case is identical. Since $f\in L^1(\Omega)$, there exists $\varepsilon>0$ such that $A_\varepsilon:=\{t\in\Omega_k:f(t)>\epsilon\}$ has positive measure. By $\sigma$-finiteness and non-atomicity of Lebesgue measure, we may choose a measurable set $E\subseteq A_\varepsilon$ such that $0<m(E)<\frac1k$.
Define $z=\frac{\chi_E}{m(E)}$.
Then $z\in S(L^1(\Omega))^+$. Moreover, since $m(E)<1/k$, we have $z(t)=\frac1{m(E)}>k\ge u(t),v(t)$ for $t\in E$.
Hence $\|u-z\|_1=\|u\|_1+1-2\int_E u$, and similarly $\|v-z\|_1=\|v\|_1+1-2\int_E v$.
Since $\|u\|_1=\|v\|_1$, we get
\begin{equation}\label{eq:3.1}
\|u-z\|_1-\|v-z\|_1
=-2\int_E(u-v)
=-2\int_E f.
\end{equation}
But $f>\varepsilon$ on $E$, so $\int_E f>0$.
Therefore $\|u-z\|_1\ne \|v-z\|_1$.
Thus, if $\|u-z\|_1 = \|v-z\|_1$ for each $z\in S(L^1(\Omega))^+$, then $u=v$.
(b) Suppose $u\ne v$, and put again $f=u-v$. As above, choose $k\in\mathbb N$, $\varepsilon>0$, and a measurable set $E\subseteq\{t\in\Omega_k:|f(t)|>\varepsilon\}$ such that $f$ has constant sign on $E$ and $0<m(E)<\frac{2-c}{2k}$.
Define $z=\frac{\chi_E}{m(E)}$.
Then $z\in S(L^1(\Omega))^+$. Also $m(E)<\frac{2-c}{2k}\le \frac1k$, so $z(t)>k\ge u(t),v(t)$ on $E$. Since $u,v\in S(L^1(\Omega))^+$, we have equality~(\ref{eq:3.1}).
Moreover, $\int_E u\le k m(E)<\frac{2-c}{2}$ implies $\|u-z\|_1=\|u\|_1+1-2\int_E u>2-2(\frac{2-c}{2})=c$.
Similarly, $\|v-z\|_1>c$.
Therefore
$$\max\{c,\|u-z\|_1\}=\|u-z\|_1\qquad\text{and}\qquad \max\{c,\|v-z\|_1\}=\|v-z\|_1.$$
Since $f$ has constant nonzero sign on $E$, we have $\int_E f\ne 0$.
Equality~(\ref{eq:3.1}) implies $\|u-z\|_1\ne \|v-z\|_1$, and hence $\max\{c,\|u-z\|_1\}\ne \max\{c,\|v-z\|_1\}$.
This proves the contrapositive, and hence proves (b).
\end{proof}

We now apply the preceding lemma to the map induced by $\Delta$ between $B(L^1(\Omega_1))^+$ and $B(L^1(\Omega_2))^+$.

\begin{lem}\label{lem:3.3}
Suppose $1<p< \infty$.
Define $T:=\mathfrak{D}_2|_{S(AC^p(\Omega_2))^+}\circ\Delta\circ\mathfrak{I}_1:B(L^1(\Omega_1))^+\to B(L^1(\Omega_2))^+$.
Then $T$ is a surjective isometry with respect to $L^1$--norm.
Moreover, there exists a unique complex--linear isometric order isomorphism $\Phi:L^1(\Omega_1)\to L^1(\Omega_2)$ such that $\Phi|_{B(L^1(\Omega_1))^+} = T$.
\end{lem}
\begin{proof}
The map $T$ is determined by the following commutative diagram:
\[
\begin{tikzcd}
S(AC^p(\Omega_1))^+ \arrow[r, "\Delta"] \arrow[d, "\mathfrak{D}_1"'] 
& S(AC^p(\Omega_2))^+ \arrow[d, "\mathfrak{D}_2"'] \\
B(L^1(\Omega_1))^+ \arrow[r,dashed, "T"'] \arrow[u, "\mathfrak{I}_1"', shift right=1.2ex]
& B(L^1(\Omega_2))^+ \arrow[u, "\mathfrak{I}_2"', shift right=1.2ex]
\end{tikzcd}
\]
For $j=1,2$, since $\mathfrak{I}_j$, $\mathfrak{D}_j|_{S(AC^p(\Omega_j)^+}$ and $\Delta$ are bijective, the map $T$ is a bijection.

Next, we prove that $\|T(u)\|_1 = \|u\|_1$ for all $u\in B(L^1(\Omega_1))^+$.
Let $u\in B(L^1(\Omega_1))^+$. 
Because $\Delta$ is an isometry and $\Delta(\mathds{1}_1)=\mathds{1}_2$, we have $\|\mathds{1}_1-\mathfrak{I}_1(u)\|_{\A,p}=\|\Delta(\mathds{1}_1)-\Delta \mathfrak{I}_1(u)\|_{\A,p}=\|\mathds{1}_2-\mathfrak{I}_2(T(u))\|_{\A,p}$. 
The last equality above follows from the fact $T=\mathfrak{I}_2^{-1}\circ\Delta\circ\mathfrak{I}_1$ being bijective.
Equivalently, we have
\begin{equation}\label{eq:3.4.1}
\left(1-\left(1-\|u\|_1^p\right)^{1/p}\right)^p+\|u\|_1^p
=
\left(1-\left(1-\|T(u)\|_1^p\right)^{1/p}\right)^p+\|T(u)\|_1^p.
\end{equation}
Define $$\varphi:[0,1]\to\mathbb{R},\qquad \varphi(t)
=
\left(1-\left(1-t^p\right)^{1/p}\right)^p+t^p.
$$
We claim that $\varphi$ is strictly increasing on $[0,1]$.
Indeed, for $0\le s<t\le 1$, we have $s^p<t^p$ and $\left(1-s^p\right)^{1/p}>\left(1-t^p\right)^{1/p}$.
Therefore $\left(1-\left(1-s^p\right)^{1/p}\right)^p
<
\left(1-\left(1-t^p\right)^{1/p}\right)^p$, and, together with $s^p<t^p$, this gives $\varphi(s)<\varphi(t)$.
Thus $\varphi$ is strictly increasing.
The equality~(\ref{eq:3.4.1}) gives $\varphi(\|u\|_1)=\varphi(\|T(u)\|_1)$.
Since $\varphi$ is strictly increasing, we obtain
\begin{equation}\label{eq:3.3.1}
\|T(u)\|_1=\|u\|_1\qquad(u\in B(L^1(\Omega_1))^+).
\end{equation}

We claim that $T$ is an isometry with respect to $L^1$--norm.
Let $u,v\in B(L^1(\Omega_1))^+$.
Since $\Delta$ is an isometry, as well as $\Delta \mathfrak{I}_1(u)=\mathfrak{I}_2(T(u))$ and $\Delta \mathfrak{I}_1(v)=\mathfrak{I}_2(T(v))$,
we have $\|\mathfrak{I}_1(u)-\mathfrak{I}_1(v)\|_{\A,p}=\|\mathfrak{I}_2(T(u))-\mathfrak{I}_2(T(v))\|_{\A,p}$.
Equivalently,
\begin{equation}\label{eq:3.3.2}
|\mathfrak{I}_1(u)(x_1)-\mathfrak{I}_1(v)(x_1)|^p+\|u-v\|_1^p=|\mathfrak{I}_2(T(u))(x_2)-\mathfrak{I}_2(T(v))(x_2)|^p+\|T(u)-T(v)\|_1^p,
\end{equation}
where $x_j$ is the left endpoint of $\Omega_j$, for $j=1,2$, respectively.
Corollary~\ref{cor:2.1} implies that
\begin{equation*}
\begin{cases}
\mathfrak{I}_1(u)(x_1) = (1 - \|u\|^p_1)^{1/p}\\
\mathfrak{I}_1(v)(x_1) = (1 - \|v\|^p_1)^{1/p}
\end{cases}
\quad \text{and}\quad
\begin{cases}
\mathfrak{I}_2(Tu)(x_2) = (1 - \|Tu\|^p_1)^{1/p}\\
\mathfrak{I}_2(Tv)(x_2) = (1 - \|Tv\|^p_1)^{1/p}
\end{cases}.
\end{equation*}
Following equality~(\ref{eq:3.3.1}) as well as above equalities, we have $|\mathfrak{I}_1(u)(x_1)-\mathfrak{I}_1(v)(x_1)|=|\mathfrak{I}_2(T(u))(x_2)-\mathfrak{I}_2(T(v))(x_2)|$.
Hence we obtain $\|T(u)-T(v)\|_1=\|u-v\|_1$ from equality~(\ref{eq:3.3.2}).
Therefore, $T:B(L^1(\Omega_1))^+\to B(L^1(\Omega_2))^+$ is a surjective isometry with respect to $L^1$--norm.

The surjective isometry of $T$ along with equality~(\ref{eq:3.3.1}) implies that $T(S(L^1(\Omega_1))^+) = S(L^1(\Omega_2))^+$ as a surjective isometry.
By the extension theorem for positive unit spheres of $L^1$-space~\cite[Theorem 11]{LNW21}, there exists a unique complex--linear isometric order isomorphism $\Phi:L^1(\Omega_1)\to L^1(\Omega_2)$ such that $\Phi|_{S(L^1(\Omega_1))^+} = T|_{S(L^1(\Omega_1))^+}$.

We now claim that $\Phi|_{B(L^1(\Omega_1))^+} = T$.
Let $u\in B(L^1(\Omega_1))^+$,
Since both $T$ and $\Phi$ are isometries from $B(L^1(\Omega_1))^+$ onto $B(L^1(\Omega_2))^+$, and $\Phi|_{S(L^1(\Omega_1))^+} = T|_{S(L^1(\Omega_1))^+}$, we have $\|\Phi(u)\|_1=\|u\|_1 = \|T(u)\|_1$, and
$\|\Phi(u) - z\|_1=\|u - \Phi^{-1}(z)\|_1 = \|T(u) - T\Phi^{-1}(z)\|_1=\|T(u) - z\|_1$ for each $z\in S(L^1(\Omega_2))^+.$
By Lemma~\ref{lem:3.2}(a), we conclude $\Phi(u) = T(u)$.
Hence, $\Phi|_{B(L^1(\Omega_1))^+} = T$. 
Since $\Phi(B(L^1(\Omega_1))^+) = B(L^1(\Omega_2))^+$, by scaling positive elements, $\Phi$ is order-preserving.
Thus, $\Phi$ is a complex--linear isometric order isomorphism.
\end{proof}

At this stage, we have gathered sufficient ingredients for Theorem~\ref{thm:main}.
The proof is organized into three cases: (1) $p=1$; (2) $1< p<\infty$; and (3) $p=\infty$.

\begin{proof}[\textbf{Proof of Theorem~\ref{thm:main}}]
\textbf{($p=1$)}
For $j=1,2$, by Proposition~\ref{prop:2.1} and Proposition~\ref{prop:2.2}, we define a complex--linear isometric order isomorphism
$$U_j:AC^1(\Omega_j)\to L^1(X_j,\mathfrak{A}_j,\mu_j),\qquad U_j(f) =f(x_j)\chi_{\{\ast_j\}}+f'\chi_{\Omega_j}.$$
From now on, we simply write $L^1(\mu_j)$ for $L^1(X_j,\mathfrak{A}_j, \mu_j)$.
Define $T:=U_2\circ\Delta\circ  U_1|_{S(AC^1(\Omega_1))^+}^{-1}:S(L^1(\mu_1))^+ \to S(L^1(\mu_2))^+$,
then $T$ is a surjective isometry.
The map $T$ is arranged as the following commutative diagram
$$
\begin{tikzcd}[column sep=large,row sep=large]
S(AC^1(\Omega_1))^+
\arrow[r,"\Delta"]
\arrow[d,"U_1"']
&
S(AC^1(\Omega_2))^+
\arrow[d,"U_2"]
\\
S(L^1(\mu_1))^+
\arrow[r, dashed,"T"]
&
S(L^1(\mu_2))^+ .
\end{tikzcd}
$$
By the extension theorem between the positive unit spheres of $L^1$-spaces~\cite[Theorem 11]{LNW21}, $T$ admits a complex--linear isometric order isomorphism $\Phi:L^1(\mu_1) \to L^1(\mu_2)$ such that $\Phi|_{S(L^1(\mu_1))^+} = T$.
The extension procedure is represented by the following commutative diagram
$$
\begin{tikzcd}[column sep=large,row sep=large]
AC^1(\Omega_1)
\arrow[r,dashed, "\widetilde{\Delta}"]
\arrow[d,"U_1"']
&
AC^1(\Omega_2)
\arrow[d,"U_2"]
\\
L^1(\mu_1)
\arrow[r,"\Phi"]
&
L^1(\mu_2).
\end{tikzcd}
$$

Define $\widetilde{\Delta}:= U_2^{-1}\circ\Phi\circ U_1:AC^1(\Omega_1)\to AC^1(\Omega_2)$.
Let $f\in S(AC^1(\Omega_1))^+$, we have $\widetilde{\Delta}(f) = U_2^{-1}\Phi U_1(f) = U_2^{-1}T U_1(f)=\Delta(f)$, and thus $\widetilde{\Delta}|_{S(AC^1(\Omega_1))^+} = \Delta$.

By Corollary~\ref{cor:2.2}, we have $\Phi(\chi_{\{*_1\}})=\chi_{\{*_2\}}$.
Therefore $\Phi$ maps the band $L^1(\Omega_1)$ onto the band $L^1(\Omega_2)$.
Define $\Lambda( u):=\left.\Phi(u\chi_{\Omega_1})\right|_{\Omega_2}$ for each $u\in L^1(\Omega_1)$.
Then $\Lambda:L^1(\Omega_1)\to L^1(\Omega_2)$ is a complex--linear isometric order isomorphism.
For $f\in AC^1(\Omega_1)$, we have $\Phi U_1(f)=f(x_1)\chi_{\{\ast_2\}}+\Lambda(f')\chi_{\Omega_2}$.
Applying $U_2^{-1}$, we obtain
$$
\widetilde{\Delta}(f)(x)=f(x_1)+\int_{x_2}^x \Lambda(f')(t)\,dt,\qquad x\in\Omega_2.
$$

Finally, if $\overline{\Delta}$ is another surjective linear isometry extending $\Delta$, then $\Psi=U_2\circ \overline{\Delta}\circ U_1^{-1}$ is a surjective linear isometry extending $T$. By the uniqueness part of~\cite[Theorem 11]{LNW21}, $\Psi=\Phi$, and hence $\overline{\Delta}=\widetilde{\Delta}$.
This proves uniqueness.

($1<p<\infty$)
Following Proposition~\ref{prop:2.1}, we define a complex--linear isometric order isomorphism
$$J_j:AC^p(\Omega_j)\to\mathbb{C}\oplus_p L^1(\Omega_j),\qquad J_j(f) = (f(x_j), f').$$
Corollary~\ref{cor:2.1} implies that there is a bijection from $S(AC^p(\Omega_j))^+$ onto $B(L^1(\Omega_j))^+$. 
By Lemma~\ref{lem:3.3}, there exists a surjective isometry $T:B(L^1(\Omega_1))^+\to B(L^1(\Omega_2))^+$, and thus an induced map $V:=J_2\circ\Delta\circ J_1^{-1}: \Sigma^+_1\to \Sigma^+_2$ defined by $V(a,u) = (a,T(u))$,
The map $V$ is arranged by the following commutative diagram
$$
\begin{tikzcd}[column sep=large,row sep=large]
S(AC^p(\Omega_1))^+
\arrow[r,"\Delta"]
\arrow[d,"J_1"']
&
S(AC^p(\Omega_2))^+
\arrow[d,"J_2"]
\\
\Sigma^+_1
\arrow[r, dashed,"V"]
&
\Sigma^+_2 .
\end{tikzcd}
$$

Lemma~\ref{lem:3.3} implies that $T$ admits a unique complex--linear isometric order isomorphism $\Phi$ such that $\Phi|_{B(L^1(\Omega_1))^+} = T$, and thus an induced extension $\widetilde{V}:\mathbb{C}\oplus_p L^1(\Omega_1)\to \mathbb{C}\oplus_p L^1(\Omega_2)$ defined by $(a,u)\mapsto(a,\Phi(u))$. 
Now define 
$\widetilde{\Delta}:=J_2^{-1}\circ\widetilde{V}\circ J_1:AC^p(\Omega_1)\to AC^p(\Omega_2).$
Let $f\in S(AC^p(\Omega_1))^+$.
Then $\widetilde{\Delta}(f) = J_2^{-1}\widetilde{V}J_1(f) = J_2^{-1}VJ_1(f)=\Delta(f)$, and hence $\widetilde{\Delta}|_{S(AC^p(\Omega_1))^+} = \Delta$.
The extension procedure is organized by the following commutative diagram
$$
\begin{tikzcd}[column sep=large,row sep=large]
AC^p(\Omega_1)
\arrow[r,dashed, "\widetilde{\Delta}"]
\arrow[d,"J_1"']
&
AC^p(\Omega_2)
\arrow[d,"J_2"]
\\
\mathbb{C}\oplus_p L^1(\Omega_1)
\arrow[r,"\widetilde{V}"]
&
\mathbb{C}\oplus_p L^1(\Omega_2).
\end{tikzcd}
$$
In particular, for each $f\in AC^p(\Omega_1)$ and $x\in\Omega_2$, we have 
$$\widetilde{\Delta}(f)(x) = J_2^{-1}\widetilde{V}(f(x_1), f')(x) = J_2^{-1}(f(x_1), \Phi(f'))(x) = f(x_1)+\int_{x_2}^x\Phi(f')(t)dt.$$
Taking $\Lambda=\Phi$ gives the desired formula.

($p=\infty$)
We adopt the following strategies.
Firstly, we observe that the set $\Sigma^+_j=\{(a,u)\in\mathbb{R}^+\oplus L_\mathbb{R}^1(\Omega_j)^+:\|(a,u)\|_{\oplus,\infty}=1\}$ is the union of 
\begin{align*}
P_j&:=\{(a,u)\in\mathbb{R}^+\oplus L_\mathbb{R}^1(\Omega_j)^+: a =1, \|u\|_1\le 1\}=\{(1,u):u\in B(L^1(\Omega_j))^+\},\quad\text{and}\\
Q_j&:=\{(a,u)\in\mathbb{R}^+\oplus L_\mathbb{R}^1(\Omega_j)^+:a\le1, \|u\|_1=1\}=\{(a,u):a\in [0,1],u\in S(L^1(\Omega_j))^+\}.
\end{align*}
Define $\Gamma:=J_2\circ\Delta\circ J_1|_{S(AC^\infty(\Omega_1))^+}^{-1}:\Sigma^+_1\to \Sigma^+_2$.
Then $\Gamma$ is a surjective isometry.
Let $(a,u)\in \Sigma^+_1$ and $\Gamma(a,u) =(b,w )\in \Sigma^+_2$.
Since $\Gamma$ is a surjective isometry, following Remark~\ref{rmk:rho}, $\rho(a,u) = \|u\|_1+1$, we have $\rho(a,u) = \rho(\Gamma(a,u))$, and thus $\|u\|_1+1 = \|w\|_1+1$.
Consequently, $\Gamma$ preserves the norm of second--coordinate.
Therefore, $\Gamma(Q_1) = Q_2$.
Moreover, $\Gamma(Q_1)=Q_2$ along with $\Gamma(\Sigma^+_1) = \Sigma^+_2$ implies that $\Gamma(\Sigma^+_1\setminus Q_1) = \Sigma^+_2\setminus Q_2$.
Indeed, $\Sigma^+_j\setminus Q_j=\{(1,u):u\in B(L^1(\Omega_j))^+\setminus S(L^1(\Omega_j))^+\}$, and thus $\overline{\Sigma^+_j\setminus Q_j} = \{(1,u):u\in B(L^1(\Omega_j))^+\} = P_j$.
Because closure property is carried by isometries,  $\Gamma(P_1) =P_2$.

Now we characterize the extension procedure of $P_j$ and $Q_j$, respectively.
Since $\Gamma(P_1) = P_2$, there exists a surjective isometry $T_P:B(L^1(\Omega_1))^+\to B(L^1(\Omega_2))^+$ such that $\Gamma|_{P_1}:P_1\to P_2$ defined by $(1,u)\mapsto (1,T_P(u))$ for each $u\in B(L^1(\Omega_1))^+$.
Note that $T_P$ is a surjective isometry with respect to $L^1$-norm and $T_P(0) = 0$ because $\Gamma|_{P_1}$ is a surjective isometry.
Then following the construction of $\Phi$ in Lemma~\ref{lem:3.3}, there exists a unique complex--linear isometric order isomorphism $\Phi:L^1(\Omega_1)\to L^1(\Omega_2)$ such that $\Phi|_{B(L^1(\Omega_1))^+} = T_P$, and hence a unique extension $\widetilde{\Gamma}(1,u) = (1,\Phi(u))$ for each $u\in L^1(\Omega_1)$.

We have characterized the unique extension of $P_j$, including on $P_j\cap Q_j$.
We now extend the procedure to $Q_j$ via $P_j\cap Q_j$.
Consider $(a,u)\in Q_1$ and $(1,z)\in P_2\cap Q_2$, we write $\Gamma(a,u)=(b,v)$.
Then $u\in S(L^1(\Omega_1))^+$ and $v,z\in S(L^1(\Omega_2))^+$.
Since $\Gamma$ is a surjective isometry, and $\Gamma(1,\Phi^{-1}(z)) = (1,z)$, we have
\begin{equation}\label{eq:thm:1}
\max\{1-a, \|u-\Phi^{-1}(z)\|_1\} = \max\{1-b,\ \|v-z\|_1\}.
\end{equation}
Also, $\Phi:L^1(\Omega_1)\to L^1(\Omega_2)$ is a surjective isometry, we have
\begin{equation}\label{eq:thm:2}
\max\{1-a, \|u-\Phi^{-1}(z)\|_1\} = \max\{1-a, \|\Phi(u)-z\|_1\},
\end{equation}
Thus, equalities~(\ref{eq:thm:1}) and (\ref{eq:thm:2}) are identified as $\max\{1-b, \|v-z\|_1\} = \max\{1-a, \|\Phi(u)-z\|_1\}$.
Since $z\in S(L^1(\Omega_2))^+$ is arbitrary, and hence choosing $z=v$, we have 
$$1-b=\max\{1-b,\|v-v\|_1\} = \max\{1-a,\|\Phi(u) - v\|_1\}\ge 1-a.$$
In the chosen $z=\Phi(u)$, we have
$$1-a = \max\{1-a, \|\Phi(u)-\Phi(u)\|_1\} = \max\{1-b,\|v-\Phi(u)\|_1\}\ge 1-b.$$
We conclude that $a=b$, and hence
$\max\{1-a, \|v-z\|_1\} = \max\{1-a, \|\Phi(u)-z\|_1\}$.
Therefore, $v=\Phi(u)$ by Lemma~\ref{lem:3.2} (b).
That is, $\Gamma(a,u)=(a,\Phi(u))$ for each $(a,u)\in Q_1$.
Since $\Sigma^+_j=P_j\cup Q_j$, the unique extension $\widetilde{\Gamma}$ obtained in $P_j$ part is written as
$$\widetilde{\Gamma}:\mathbb{C}\oplus_\infty L^1(\Omega_1)\to\mathbb{C}\oplus_\infty L^1(\Omega_2)\qquad\text{by}\qquad \widetilde{\Gamma}(a, u) = (a,\Phi(u))\quad ((a,u)\in\mathbb{C}\oplus_\infty L^1(\Omega_1)),$$
such that $\widetilde{\Gamma}|_{\Sigma^+_1} = \Gamma$.

Define $\widetilde{\Delta}:= J_2^{-1}\circ\widetilde{\Gamma}\circ J_1: AC^\infty(\Omega_1)\to AC^\infty(\Omega_2)$.
Let $f\in S(AC^\infty(\Omega_1))^+$, then we have $\widetilde{\Delta}(f) = J_2^{-1}\widetilde{\Gamma}J_1(f) = J_2^{-1}\Gamma J_1(f) = \Delta(f)$, and thus $\widetilde{\Delta}|_{S(AC^\infty(\Omega_1))^+} = \Delta$.
In particular, for $f\in AC^\infty(\Omega_1)$,
$$\widetilde{\Delta}(f)(x) = J_2^{-1}\widetilde{\Gamma}(f(x_1),f')(x) = J_2^{-1}(f(x_1), \Phi(f'))(x) = f(x_1)+\int_{x_2}^x \Phi(f')(t)dt\qquad x\in\Omega_2.$$
Taking $\Lambda=\Phi$ gives the desired formula.

For each $1< p\le \infty$, the uniqueness of $\widetilde{\Delta}$ is proved as follows.
Suppose $\overline{\Delta}$ is another extension, then $J_2\circ\overline{\Delta}\circ J_1^{-1}$ is another extension of $\Gamma$ ($V$, in $p\neq\infty$), contradicting the uniqueness of~\cite[Theorem 11]{LNW21}.
Therefore, $\widetilde{\Delta}$ is the unique complex--linear isometric order isomorphism extension of $\Delta$.
\end{proof}

The proof of Theorem~\ref{thm:main} highlights two features of the present approach.

First, we emphasize the separation into three cases.
The three cases $p=1$, $1<p<\infty$, and $p=\infty$ require different arguments because the relevant metric properties change at the two endpoints $p=1$ and $p=\infty$.
For $p=1$, the radius function $\rho$ used in Lemma~\ref{lem:3.1} does not distinguish the constant function $\mathds{1}$ from the other elements of $S(AC^1(\Omega))^+$.
We realize $AC^1(\Omega)$ as an $L^1$--space instead.
For $p=\infty$, the restriction $\mathfrak{D}|_{S(AC^\infty(\Omega))^+}$ loses injectivity, and hence we decompose $\Sigma^+=P\cup Q$.
On $P$, the scalar coordinate is fixed, and hence the derivative coordinate determines the point uniquely.
The set $Q$ contains precisely the points for which the derivative alone leaves the scalar coordinate undetermined.
This decomposition therefore allows the two parts of $\Sigma^+$ to be treated separately.

Second, we overcome the structural obstructions that prevent the argument of~\cite{C1LIP} from applying to $AC^p(\Omega)$.
The proof replaces the convex--body argument by a different extension mechanism.
For $1<p\leq\infty$, we first apply the extension theorem for positive unit spheres of $L^1$--spaces~\cite[Theorem 11]{LNW21} and then use the metric separation properties established in Lemma~\ref{lem:3.2}.
When $1<p<\infty$, Lemma~\ref{lem:3.2}(a) shows that the resulting complex--linear isometric order isomorphism agrees with the induced isometry on the whole positive unit ball.
When $p=\infty$, Lemma~\ref{lem:3.2}(b) shows that the distance relations with points of $P\cap Q$ uniquely determine the remaining $L^1$--coordinate.
Thus, in both cases, the extension is obtained from the $L^1$--sphere extension theorem together with the separation principle of Lemma~\ref{lem:3.2}.

\section{Application to phase--isometries}
In this section, the main Theorem~\ref{thm:main} will be utilized in phase--isometries.
Let $A, B$ be non--empty subsets of normed spaces $E_1$ and $E_2$, respectively.
A map $T:A\to B$ is a \textit{phase--isometry} if 
$$\Big\{\|x-\alpha y\|_{E_1}:\alpha\in\mathbb{T}\Big\} = \Big\{\|T(x) - \alpha T(y)\|_{E_2}:\alpha\in\mathbb{T}\Big\}\qquad (x,y\in A).$$
Phase-isometries have been investigated in~\cite{SSD23,IOT21}.
As an application of Theorem~\ref{thm:main}, we prove the following result.
\begin{thm}
    Suppose $T:S(AC^p(\Omega_1))^+\to S(AC^p(\Omega_2))^+$ is a surjective phase--isometry.
    Then there exists a unique complex--linear isometric order isomorphism $\widetilde{T}:AC^p(\Omega_1)\to AC^p(\Omega_2)$ such that $\widetilde{T}|_{S(AC^p(\Omega_1))^+} = T$.
    In particular, 
    $$\widetilde{T}(f)(x) = f(x_1)+ \int_{x_2}^x \Lambda (f')(t) dt,\qquad x\in\Omega_2,\quad f\in AC^p(\Omega_1),$$
    where $\Lambda:L^1(\Omega_1)\to L^1(\Omega_2)$ is the unique complex--linear isometric order isomorphism determined by $T$.
\end{thm}
\begin{proof}
For $f,g\in S(AC^p(\Omega_1))^+$, and each $1\le p\le \infty$, we observe that $\|f-g\|_{\A,p}\le \|f-\alpha g\|_{\A,p}$ for all $\alpha\in\mathbb{T}$.
To show this, firstly, we have $|f(x_1) - \alpha g(x_1)|^2 - |f(x_1) - g(x_1)|^2 = 2f(x_1)g(x_1)(1-\mathrm{Re}(\alpha))\ge 0$, and hence $|f(x_1)-g(x_1)|\le |f(x_1)- \alpha g(x_1)|$.
The same argument on their derivative gives $|f'(t) - g'(t)|\le |f'(t) - \alpha g'(t)|$ for almost every $t\in\Omega_1$.
By integrating over $\Omega_1$, we have $\|f'-g'\|_1\le \|f'-\alpha g'\|_1$.
Substituting into $p$--norm~(\ref{p_norm}), we obtain the inequality $\|f-g\|_{\A, p}\le \|f-\alpha g\|_{\A, p}$.

Similarly, $\|T(f)-T(g)\|_{\A,p}\le \|T(f) -\alpha T(g)\|_{\A,p}$ for such $f$ and $g$.
Thus,
\begin{align*}
    \|T(f)-T(g)\|_{\A,p} &= \min\{\|T(f)-\alpha T(g)\|_{\A,p}:\alpha\in\mathbb{T}\}\\
    &=\min\{\|f-\alpha g\|_{\A,p}:\alpha\in\mathbb{T}\}\\
    &=\|f-g\|_{\A,p}.
\end{align*}
The second equality above follows from the fact that $T$ is a phase-isometry.
Therefore, $T$ is a surjective isometry.
By Theorem~\ref{thm:main}, the desired conclusion follows.
\end{proof}

\section*{Acknowledgments}
The author would like to thank Professor Takeshi Miura for his valuable advice.
The author was supported by the National Science and Technology Council, Taiwan (NSTC), under Grant No. 115--2917--I--110--009.

\end{document}